\documentclass[10pt]{amsart}
\usepackage{ amssymb }
\usepackage[margin=0.9in]{geometry}
\usepackage{ amsmath}
\usepackage[mathscr]{eucal}
\DeclareSymbolFont{largesymbolsA}{U}{txexa}{m}{n}
\DeclareMathSymbol{\varprod}{\mathop}{largesymbolsA}{16}
\newtheorem{theorem}{Theorem}[section]
\newtheorem{conj}{Conjecture}[section]
\newtheorem{lem}[theorem]{Lemma}
\newtheorem{prop}[theorem]{Proposition}
\newtheorem{cor}{Corollary}[theorem]
\theoremstyle{definition}

\theoremstyle{remark}

\DeclareMathOperator{\sign}{sgn}
\numberwithin{equation}{section}

\begin{document}
\author{Ankush Goswami}
\address{Research Institute for Symbolic Computation (RISC), JKU, Linz.}
\email{ankushgoswami3@gmail.com, ankush.goswami@risc.jku.at}
\title[On sums related to Borwein conjectures]{On sums of coefficients of polynomials related to the Borwein conjectures}
 \author{Venkata Raghu Tej Pantangi}
 \address{Department of Mathematics, Southern University of Science and Technology (SUSTECH), Shenzhen, China.}
\email{pvrt1990@gmail.com, pantangi@sustech.edu.cn}

\thanks{}


\date{}

\begin{abstract}
 Recently, Li (Int. J. Number Theory 2020) obtained an asymptotic formula for a certain partial sum involving coefficients for the polynomial in the First Borwein conjecture. As a consequence, he showed the positivity of this sum. His result was based on a sieving principle discovered by himself and Wan (Sci. China. Math. 2010). In fact, Li points out in his paper that his method can be generalized to prove an asymptotic formula for a general partial sum involving coefficients for any prime $p>3$. In this work, we extend Li's method to obtain asymptotic formula for several partial sums of coefficients of a very general polynomial. We find that in the special cases $p=3, 5$, the signs of these sums are consistent with the three famous Borwein conjectures. Similar sums have been studied earlier by Zaharescu (Ramanujan J. 2006) using a completely different method. We also improve on the error terms in the asymptotic formula for Li and Zaharescu. Using a recent result of Borwein (JNT 1993), we also obtain an asymptotic estimate for the maximum of the absolute value of these coefficients for primes $p=2, 3, 5, 7, 11, 13$ and for $p>15$, we obtain a lower bound on the maximum absolute value of these coefficients for sufficiently large $n$. 
\end{abstract}
\maketitle
\section{Introduction}
In $1990$, Peter Borwein (see \cite{And-Bor}) empirically discovered quite a number of mysteries involving sign patterns of coefficients of certain polynomials. The most easily stated are the following:
\begin{conj}[First Borwein conjecture]\label{bor1}
For the polynomials $A_n(q), B_n(q)$ and $C_n(q)$ defined by 
\begin{eqnarray*}
\prod_{j=1}^n(1-q^{3j-2})(1-q^{3j-1})=A_n(q^3)-qB_n(q^3)-q^2C_n(q^3)
\end{eqnarray*}
each has non-negative coefficients. 
\end{conj}
\begin{conj}[Second Borwein conjecture]\label{bor2}
For the polynomials $\alpha_n(q), \beta_n(q)$ and $\gamma_n(q)$ defined by 
\begin{eqnarray*}
\prod_{j=1}^n(1-q^{3j-2})^2(1-q^{3j-1})^2=\alpha_n(q^3)-q\beta_n(q^3)-q^2\gamma_n(q^3)
\end{eqnarray*}
each has non-negative coefficients. 
\end{conj}
\begin{conj}[Third Borwein conjecture]\label{bor3}
For the polynomials $\nu_n(q), \phi_n(q), \chi_n(q)$, $\psi_n(q)$ and $\omega_n(q)$ defined by 
\begin{eqnarray*}
\prod_{j=1}^n(1-q^{5j-4})(1-q^{5j-3})(1-q^{5j-2})(1-q^{5j-1})=\nu_n(q^5)-q\phi_n(q^5)-q^2\chi_n(q^5)-q^3\psi(q^5)-q^4\omega_n(q^5)
\end{eqnarray*}
each has non-negative coefficients. 
\end{conj}
Recently, Wang \cite{Wang} gave an analytic proof of the First Borwein conjecture using saddle point method. His proof, besides other things, relied on a formula of Andrews \cite[Theorem 4.1]{And-Bor} and the following recursive relations \cite[Theorem 3.1]{And-Bor}:  
\begin{eqnarray*}
A_n(q)&=&(1+q^{2n-1})A_{n-1}(q)+q^n B_{n-1}(q)+q^nC_{n-1}(q),\\
 B_{n}(q)&=&q^{n-1}A_{n-1}(q)+(1+q^{2n-1})B_{n-1}(q)-q^nC_{n-1}(q),\\
 C_{n}(q)&=&q^{n-1}A_{n-1}(q)+q^{n-1}B_{n-1}(q)-(1+q^{2n-1})C_{n-1}(q).
\end{eqnarray*}
Let $p\geq 3$ be a prime and $s, n \in \mathbb{N}$. Consider the polynomial
\begin{eqnarray}
T_{p,s,n}(q):=\prod_{j=0}^n\prod_{k=1}^{p-1}(1-q^{pj+k})^s.
\end{eqnarray}
For $s=1$, Borwein \cite{Bor} obtained an asymptotic estimate for $\left\|T_{p,1,n}(q)\right\|_{|q|=1}=\sup_{|q|=1}|T_{p,1,n}(q)|$ when $p=2, 3, 5, 7, 11$ and $13$. However for $p>15$, he obtained an asymptotic lower bound for this quantity. 

It is clear that $d_{p,s,n}:=\;$deg$\;T_{p,s,n}=p(p-1)s(n+1)^2/2$. Define the coefficients $t_{i,p,s}$ by
\begin{eqnarray}\label{Tpq}
T_{p,s,n}(q)&:=&\sum\limits_{i=0}^{d_{p,s,n}}t_{i,p,s}q^{i}\notag\\&=&T_{0,p,s,n}(q^p)+qT_{1,p,s,n}(q^p)+\cdots+q^{p-1}T_{p-1,p,s,n}(q^p),
\end{eqnarray}
where $T_{i,p,s,n}(q)\in\mathbb{Z}[q]$. Given a polynomial $f(x)$, by $[x^{j}]f(x)$, we denote the coefficient of $x^{j}$ in $f(x)$. Let $a, d\in\mathbb{Z}$. In what follows, assume $p\mid a$ and let $S_{a,d,j,p}$ denote the arithmetic progression
\begin{eqnarray}
S_{a,d,j,p}:=\left\{am+d: m\in\mathbb{Z}\right\}, \;\mbox{with}\;d\equiv j \pmod{p}.
\end{eqnarray}
Put $a=p\ell$ and consider the following finite sum of coefficients over $S_{a,d,j,p}$:
\begin{eqnarray}\label{absums}
\sum_{\substack{i\geq 0\\i\in S_{p\ell,d,j,p}}}t_{i,p,s}=\sum_{\substack{i\geq 0\\i\in S_{p\ell,d,j,p}}}[q^{i-j}]T_{j,p,s,n}(q^p).
\end{eqnarray}
In \cite[pp. 98, Theorem 1]{Z}, Zaharescu obtained an asymptotic formula for the sum in (\ref{absums}) when $\ell$ is an odd prime $\leq n+1$ and $\ell\neq p$. As a result, when $\ell\leq c(n+1)$ with $0<c<1$, he showed positivity (resp. negativity) of the sum in (\ref{absums}) when $j=0$ (resp. $j\neq 0$) for large $n$. 

As Zaharescu points out in his paper, it is interesting to obtain positivity (or negativity) of the above sum for larger values of $\ell$. When $\ell\gg (n+1)^2$ (with implied constant larger than $1$), one can isolate each individual terms in the sum (\ref{absums}). We note here that the main disadvantage of his asymptotic formula is the error term, which is large. This forces him to choose a $\ell\ll n+1$ which ensures that the main term is bigger than the error term, thereby showing postivity or negativity of the sums.   

For $(p,s,\ell,j)=(3,1,n+1,0)$, Li \cite{Li} obtained an asymptotic formula for the sum in (\ref{absums}) using a new sieve technique discovered by himself and Wan \cite{Li-Wan}. If we denote by $t_{i,3,1}=a_{i}$, then Li proved that
\begin{theorem}[Li]\label{Li}
For $0\leq j\leq (n+1)$ we have
\begin{eqnarray*}
\sum_{\ell'=0}^n a_{3j+3\ell'(n+1)}=\dfrac{2\cdot 3^{n}}{n+1}(1+o(1)).
\end{eqnarray*}
In particular, we have
\begin{eqnarray*}
\sum_{\ell'=0}^n a_{3j+3\ell(n+1)}>0.
\end{eqnarray*}
\end{theorem}
Indeed, the error term in Li's asymptotic formula \cite[pp. 4, Theorem 1.5]{Li} is much better than Zaharescu's which enabled him to prove the positivity of the sum in Theorem \ref{Li}.

The purpose of this paper is to extend Li's results by obtaining asymptotic formula for the sums in (\ref{absums}) in the case $\ell=n+1$ for all $p, s, j$. As a consequence, we obtain positivity (or negativity) of the sums in (\ref{absums}) for large $n$. Thus, for $p=3,5$, we obtain asymptotic formula for the partial sums of coefficients involving polynomials in Conjectures \ref{bor1}-\ref{bor3}. This in turn shows that the sums are positive (or negative) for all $n>0$. (see Corollaries \ref{sumbor1}--\ref{sumbor3}). We also improve on the error terms in Li's and Zaharescu's asymptotic formula. Using a recent result of Borwein \cite{Bor}, we also obtain an asymptotic estimate for the maximum absolute coefficients of $T_{p,s,n}(q)$ only in the case $p=2, 3, 5, 7, 11, 13$; however for $p>15$ we obtain an asymptotic lower bound for the maximum absolute coefficients. 

This paper is organized as follows. In Section \ref{NCB} we introduce a few notations, conventions and do some basic counting. In Section \ref{MR} we state our main results. In Section \ref{LWs} we recall Li and Wan's \cite{Li-Wan} sieving principle and also establish a few basic results. Finally in Section \ref{Ps} we obtain the proofs of our main results.
\section{Acknowledgement}
The research of the first author was supported by grant SFB F50-06 of the Austrian Science Fund (FWF). 
The authors thank George Andrews, Peter Paule, Qing Xiang and Cai-Heng Li for their feedback. We also thank the anonymous referee for valuable suggestions and feedback.
\section{Notation, Conventions and Basic Counting}\label{NCB}
Let $n\in\mathbb{N}$ and $p\geq 3$ be a prime. Set $N_p=(n+1)p$ and $D_p=\{1,2,\ldots p-1, p+1,\ldots,2p-1,\ldots, pn+1,\ldots, pn+p-1\}$. We define the following: 
\begin{eqnarray*}
\mathscr{C}_{e,p,s}(j,n)&:=&\#\left\{\varprod_{i=1}^s V_i\subset D_p^s:\sum_{i=1}^s|V_i|\equiv 0 \;(\mbox{mod}\;2), \sum_{i=1}^s\sum\limits_{x_{v_i}\in V_i} x_{v_i}= j\right\},\notag\\
\mathscr{C}_{o,p,s}(j,n)&:=&\#\left\{\varprod_{i=1}^s V_i \subset D_p^s:\sum_{i=1}^s|V_i|\equiv 1\;(\mbox{mod}\;2), \sum_{i=1}^s\sum\limits_{x_{v_i}\in V_i} x_{v_i}= j\right\}.
\end{eqnarray*}
It is now apparent that
\begin{equation}\label{co}
 t_{j,p,s}=\mathscr{C}_{e,p,s}(j,n)-\mathscr{C}_{o,p,s}(j,n).
\end{equation}
We note that in the case $s=1$, $\mathscr{C}_{e,p,1}(j,n)$ (respectively $\mathscr{C}_{o,p,1}(j,n)$) counts the number of partitions of $j$ into an even (respectively odd) number of distinct non-multiples of $p$. 

As in \cite{Li}, we shift the problem to that of counting the size of certain subsets of the group $G=\mathbb{Z}_{N_p}$. We note that $G\setminus D_{p}$ is a subgroup of index $p$. Given $0\leq k_1, k_2,\cdots,k_s \leq |D_{p}|$ and $0\leq b < N_p$, define
\begin{eqnarray*}
M_{p,s,n}(k_1,k_2,\cdots,k_s;b):=\#\left\{\varprod_{i=1}^sV_i \subset D_p^s:|V_1|=k_1,\cdots,|V_s|=k_s, \sum_{i=1}^s\sum\limits_{x_{v_i}\in V_i} x_{v_i}\equiv b \;(\mbox{mod}\;N_p)\right\}, 
\end{eqnarray*}
and set
\begin{eqnarray*}
M_{p,s,n}(b)=\sum_{0\leq k_1,k_2,\cdots,k_s\leq |D_{p}|}(-1)^{k_1+k_2+\cdots+k_s}M_{p,s,n}(k_1,k_2,\cdots,k_s;b).
\end{eqnarray*}
From (\ref{absums}) and (\ref{co}), we see that if $b\equiv j\;(\mbox{mod}\;p)$ then the following are equivalent:
\begin{equation}\label{ps}
M_{p,s,n}(b)=\sum_{\substack{0\leq i\leq d_{p,s,n}\\i\in S_{N_{p},b,j,p}}}t_{i,p,s}=\sum_{0\leq \ell\leq B_{p,s}(b)}t_{N_p\ell+b,p,s}=\sum_{0\leq \ell\leq B_{p,s}(b)}[q^{b-j+\ell N_p}]T_{j,p,s,n}(q^p),
\end{equation}
where $B_{p,s}(b):=\lfloor(d_{p,s,n}-b)/N_p\rfloor$. For ease of notation, we will mostly use the second sum in (\ref{ps}) for $M_{p,s,n}(b)$. 

We next introduce a few more notations. Let $(x)_k:=x(x-1)(x-2)\cdots(x-k+1)$ denote the falling factorial. Let $\hat{G}$ be the set of complex-valued linear characters of $G$. By $\psi_{0}$, we denote the trivial character in $\hat{G}$. Let $X_{p,k}=D_p^k$ and $\overline{X}_{p,k}$ denote the subset of all tuples in $D_p^{k}$ with distinct coordinates. 
\section{Main results}\label{MR}
Our main results are below. 
\begin{theorem}\label{main1}
With $M_{p,s,n}(b)$ defined as in \emph{(\ref{ps})} and $b\in\mathbb{Z}_{N_p}$ we have
\begin{eqnarray*}
\left|M_{p,s,n}(b)-\dfrac{\varSigma_{p,s,n,b}}{n+1}\right|\leq p^{s(n+1)/2},
\end{eqnarray*}
where 
\begin{eqnarray*}
\varSigma_{p,s,n,b}=\left\{\begin{array}{cc}
(p-1)\cdot p^{s(n+1)-1},&\emph{if $p \mid b$,}\\
-p^{s(n+1)-1},&\emph{otherwise.}\end{array}\right.
\end{eqnarray*}
\end{theorem}
\begin{theorem}\label{main2}
For a fixed prime $p\geq 3$ and $b\in\mathbb{Z}_{N_p}$, define
\begin{eqnarray*}n_{p,s,b}:=\left\{\begin{array}{cc}
\inf\left\{n\in\mathbb{N}: (p-1)p^{s(n+1)/2-1}>n+1\right\},&\emph{if $p\mid b$,}\\
\inf\left\{n\in\mathbb{N}: p^{s(n+1)/2-1}>n+1\right\},&\emph{otherwise}.
\end{array}\right.
\end{eqnarray*}
Then for all $n\geq n_{p,s,b}$ we have
\begin{eqnarray*}
M_{p,s,n}(b)\left\{\begin{array}{cc}
>0,& \emph{if $p\mid b$,}\\
<0,&\emph{otherwise}.\end{array}\right.
\end{eqnarray*}
\end{theorem}
For $p=3$, Li's theorem \cite[pp. 4, Prop. 1.6]{Li} shows that $M_{3,1,n}(b) (=M_{3,n}(b))>0$ when $b\equiv 0\;(\mbox{mod\;3})$ for all $n>0$. For $p=3$, when $b\not\equiv 0\;(\mbox{mod}\;3)$, we have
\begin{theorem}\label{main3}
Let $b\equiv 1, 2\;(\emph{mod}\;3)$ with $b\in\mathbb{Z}_{N_3}$. Then
\begin{eqnarray*}
M_{3,1,n}(b)=-\dfrac{3^{n}}{n+1}(1+o(1)).
\end{eqnarray*}
In particular, $M_{3,1,n}(b)<0$ for all $n>0$.
\end{theorem}
\begin{theorem}\label{main4}
For $p=3, s=2$ and $b\in\mathbb{Z}_{N_3}$ we have
\begin{eqnarray*}
M_{3,2,n}(b)=\left\{\begin{array}{ccc}
\dfrac{2\cdot 3^{2n+1}}{n+1}(1+o(1)),&\emph{if $3\mid b$},\\
\mbox{}&\\
-\dfrac{3^{2n+1}}{n+1}(1+o(1)), &\emph{otherwise}.
\end{array}\right.
\end{eqnarray*}
In particular, $M_{3,2,n}(b)>0$ \emph{(}resp. $M_{3,2,n}(b)<0$\emph{)} when $b\equiv 0 \pmod 3$ \emph{(}resp. $b\not\equiv 0 \pmod 3$\emph{)} for all $n>0$.
\end{theorem}
\begin{theorem}\label{main5}
For $p=5, s=1$ and $b\in\mathbb{Z}_{N_5}$ we have
\begin{eqnarray*}
M_{5,1,n}(b)=\left\{\begin{array}{ccc}
\dfrac{4\cdot 5^{n}}{n+1}(1+o(1)),&\emph{if $5\mid b$},\\
\mbox{}&\\
-\dfrac{5^{n}}{n+1}(1+o(1)), &\emph{otherwise}.
\end{array}\right.
\end{eqnarray*}
In particular, $M_{5,1,n}(b)>0$ \emph{(}resp. $M_{5,1,n}(b)<0$\emph{)} when $b\equiv 0 \pmod 5$ \emph{(}resp. $b\not\equiv 0 \pmod 5$\emph{)} for all $n>0$.
\end{theorem}
In view of (\ref{Tpq}), we immediately deduce the following from Theorem \ref{main1}:
\begin{cor}
For $p\geq 3$ and $b\in\mathbb{Z}_{N_p}$, let $b\equiv j\;(\emph{mod}\;p)$. Then we have
\begin{eqnarray*}
\sum_{0\leq \ell\leq B_{p,s}(b)} t_{b+\ell N_p,p,s}=\sum_{0\leq \ell\leq B_{p,s}(b)} [q^{b-j+\ell N_p}]T_{j,p,s,n}(q^p)=\dfrac{\varSigma_{p,s,n,b}}{n+1}(1+o(1)),
\end{eqnarray*}
where $\varSigma_{p,s,n,b}$ is as in Theorem \emph{\ref{main1}} and $B_{p,s}(b)=\lfloor(d_{p,s,n}-b)/N_p\rfloor.$
\end{cor}
In particular, noting the fact that the polynomials $T_{j,p,s,n}(q)$ are the polynomials in the first three Borwein conjectures for suitable choices of $j, p$ and $s$ we have, in view of Theorems \ref{main3}-\ref{main5} the following:
\begin{cor}\label{sumbor1}
For $p=3, s=1$ and a fixed $b\in\mathbb{Z}_{N_3}$, let $b=3u+j$ be such that $j\equiv 1, 2\;(\emph{mod}\;3)$. Then we have
\begin{eqnarray*}
0>\sum_{0\leq \ell\leq B_{3,1}(b)} t_{3u+1+\ell N_3,3,1}=\sum_{0\leq \ell\leq B_{3,1}(b)} [q^{3u+\ell N_3}]B_{n}(q^3)=-\dfrac{3^{n}}{n+1}(1+o(1)),\\
0>\sum_{0\leq \ell\leq B_{3,1}(b)} t_{3u+2+\ell N_3,3,1}=\sum_{0\leq \ell\leq B_{3,1}(b)} [q^{3u+\ell N_3}]C_{n}(q^3)=-\dfrac{3^{n}}{n+1}(1+o(1)).
\end{eqnarray*}
\end{cor}
\begin{cor}\label{sumbor2}
For $p=3, s=2$ and a fixed $b\in\mathbb{Z}_{N_3}$, let $b=3u+j$ be such that $j\equiv 0, 1, 2\;(\emph{mod}\;3)$. Then we have
\begin{eqnarray*}
0<\sum_{0\leq \ell\leq B_{3,2}(b)} t_{3u+\ell N_3,3,2}=\sum_{0\leq \ell\leq B_{3,2}(b)} [q^{3u+\ell N_3}]\alpha_{n}(q^3)=\dfrac{2\cdot 3^{2n+1}}{n+1}(1+o(1)),\\
0>\sum_{0\leq \ell\leq B_{3,2}(b)} t_{3u+1+\ell N_3,3,2}=\sum_{0\leq \ell\leq B_{3,2}(b)} [q^{3u+\ell N_3}]\beta_{n}(q^3)=-\dfrac{3^{2n+1}}{n+1}(1+o(1)),\\
0>\sum_{0\leq \ell\leq B_{3,2}(b)} t_{3u+2+\ell N_3,3,2}=\sum_{0\leq \ell\leq B_{3,2}(b)} [q^{3u+\ell N_3}]\gamma_{n}(q^3)=-\dfrac{3^{2n+1}}{n+1}(1+o(1)).
\end{eqnarray*}
\end{cor}
\begin{cor}\label{sumbor3}
For $p=5, s=1$ and $b\in\mathbb{Z}_{N_5}$, let $b\equiv j\;(\emph{mod}\;5)$. Then we have
\begin{eqnarray*}
\sum_{0\leq \ell\leq B_{5,1}(b)} t_{b+\ell N,5,1}=\sum_{0\leq \ell\leq B_{5,1}(b)} [q^{b-j+\ell N}]T_{j,5,1,n}(q^5)=\left\{\begin{array}{ccc}
\dfrac{4\cdot 5^{n}}{n+1}(1+o(1))>0,&\emph{if $5|b$},\\
\mbox{}&\\
-\dfrac{5^{n}}{n+1}(1+o(1))<0, &\emph{otherwise}.
\end{array}\right.
\end{eqnarray*}
where $T_{0,5,1,n}(q)=\nu_n(q),\; T_{1,5,1,n}(q)=\phi_n(q), \;T_{2,5,1,n}(q)=\chi_n(q),\; T_{3,5,1,n}(q)=\psi_n(q),\; T_{4,5,1,n}(q)=\omega_n(q)$ are the polynomials in Conjecture \emph{\ref{bor3}}.
\end{cor}
\begin{theorem}\label{Borw1}
Let $p=2, 3, 5, 7, 11, 13$ and $n$ be sufficiently large. Then we have
\begin{eqnarray*}
\max_{i}|t_{i,p,s}|=p^{s(n+1)+O(\log n)}.
\end{eqnarray*}
\end{theorem}
\begin{theorem}\label{Borw2}
Let $p>15$ and $n$ be sufficiently large. Then we have
\begin{eqnarray*}
\max_{i}|t_{i,p,s}|\gtrsim \dfrac{(1.219\cdots)^{s(p-1)(n+1)}}{sp^2n^2}>\dfrac{p^{s(n+1)-2}}{sn^2}.
\end{eqnarray*}
\end{theorem}
\section{Li-Wan Sieve}\label{LWs}
The quantity $M_{p,s,n}(k_1,k_2,\cdots,k_s;b)$ is the number of certain type of subsets of $D_p^s$. As in \cite{Li} we apply some elementary character theory to estimate it. 

We note that 
\begin{eqnarray*}
\rho:=\sum\limits_{\psi \in \hat{G}}\psi
\end{eqnarray*}
is the regular character of $G$. It is well-known that $\rho(g)=0$ for all $g\in G\setminus \{0\}$, and that $\rho(0)=|G|=N_p$.
Given $0< r \leq |D_{p}|$, a character $\psi \in \hat{G}$, and $\bar{x}=(x_{1},\ldots ,\ x_{r})$, we set
\begin{eqnarray*}
\prod_{i=1}^r\psi(x_{i}):=f_{\psi}(\bar{x}),\ \text{and}\hspace{1cm}  
\mathcal{S}(\bar{x}):=\sum_{i=1}^r x_i.
\end{eqnarray*}
Let $Y_{p,s}^{k_1,k_2,\cdots,k_s}$ denote the cartesian product $\displaystyle\prod_{i=1}^s \overline{X}_{p,k_i}$. Then we have
\begin{eqnarray}
&&k_1!k_2!\cdots k_s!M_{p,s,n}(k_1,k_2,\cdots,k_s;b)=\dfrac{1}{N_p}\sum\limits_{(\bar{x}_1,\bar{x}_2,\cdots,\bar{x}_s)\in Y_{p,s}^{k_1,k_2,\cdots,k_s}} \sum \limits_{\psi \in \hat{G}} \psi(\mathcal{S}(\bar{x}_1)+\mathcal{S}(\bar{x}_2)+\cdots+\mathcal{S}(\bar{x}_s)-b)\notag\\
&&\hspace{3cm}=N_p^{-1}\prod_{i=1}^s\left(\dfrac{(p-1)N_p}{p} \right)_{k_{i}}+N_p^{-1}\sum\limits_{(\bar{x}_1,\bar{x}_2,\cdots,\bar{x}_s)\in Y_{p,s}^{k_1,k_2,\cdots,k_s}} \sum \limits_{\psi_0\neq \psi \in \hat{G}}\psi^{-1}(b)\prod_{i=1}^s\psi(\mathcal{S}(\bar{x}_i)).\notag
\end{eqnarray}
In the right-hand side above we interchange the sums to get,
\begin{eqnarray}
&&k_1!k_2!\cdots k_s!M_{p,s,n}(k_1,k_2,\cdots,k_s;b)=N_p^{-1}\prod_{i=1}^s\left(\dfrac{(p-1)N_p}{p} \right)_{k_{i}}\notag\\&&\hspace{3cm}+N_p^{-1}\sum \limits_{\psi_0\neq \psi \in \hat{G}}\psi^{-1}(b)\sum\limits_{(\bar{x}_1,\bar{x}_2,\cdots,\bar{x}_s)\in Y_{p,s}^{k_1,k_2,\cdots,k_s}} \prod_{i=1}^sf_\psi(\bar{x}_i)\notag\\
&&\hspace{4cm}=N_p^{-1}\prod_{i=1}^s\left(\dfrac{(p-1)N_p}{p} \right)_{k_{i}}+N_p^{-1}\sum \limits_{\psi_0\neq \psi \in \hat{G}}\psi^{-1}(b)\prod_{i=1}^s\left(\sum\limits_{\bar{x}_i\in \overline{X}_{p,k_i}} f_\psi(\bar{x}_i)\right)
\end{eqnarray}
For a $Y \subset X_{p,k}$ and a character $\psi \in \hat{G}$, set $F_{\psi}(Y):=\sum\limits_{\bar{y} \in Y} f_{\psi}(\bar{y})$. We now have 
\begin{eqnarray}\label{rc}
\ \ \ \ \ \ \ \ \ \ \ k_1!k_2!\cdots k_s!M_{p,s,n}(k_1,k_2,\cdots,k_s;b)=N_p^{-1}\prod_{i=1}^s\left(\dfrac{(p-1)N_p}{p} \right)_{k_{i}}+N_p^{-1}\sum \limits_{\psi_0\neq \psi \in \hat{G}}\psi^{-1}(b)\prod_{i=1}^sF_{\psi}(\overline{X}_{p,k_i})
\end{eqnarray}
We now estimate sums of the form $F_{\psi}(\overline{X}_{p,k})$.
The symmetric group $S_{k}$ acts naturally on $X_{p,k}=D_p^{k}$. Let $\tau \in S_{k}$ be a permutation whose cycle decomposition is
\begin{eqnarray*}
\tau=(i_1i_2\cdots i_{a_1})(j_1j_2\cdots j_{a_2})\cdots (\ell_1\ell_2\cdots \ell_{a_s})
\end{eqnarray*}
where $a_i\geq 1, 1\leq i\leq s$. We define 
\begin{eqnarray*}
X_{p,k}^{\tau}:=\left\{(x_1,x_2,\cdots,x_k)\in X_{p,k}: x_{i_1}=\cdots=x_{i_{a_1}}, \cdots,x_{\ell_1}=\cdots=x_{\ell_{a_s}} \right\}.
\end{eqnarray*}
In other words, $X_{p,k}^{\tau}$ is the set of elements in $X_{p,k}$ fixed under the action of $\tau$. Let $C_{k}$ be a set of conjugacy class representatives of $S_{k}$.  
Let us denote by $C(\tau)$ the number of elements conjugate to $\tau$. 
Now for any $\tau \in S_{k}$, we have  $\tau(X_{p,k})=X_{p,k}$. We note that for any pair $\tau$, $\tau'$ of conjugate permutations, and for any $\psi \in \hat{G}$, we have $F_{\psi}(\overline{X}^\tau_{n,k})=F_{\psi}(\overline{X}^{\tau'}_{n,k_i})$. That is, according to the definitions in \cite{Li-Wan}, $X_{p,k}$ is symmetric and $f_\psi$ is normal on $X$. Thus we have the following result which is essentially \cite[Proposition 2.8]{Li-Wan}.
\begin{prop}\label{LWS}
We have
\begin{equation*}
F_{\psi}(\overline{X}_{p,k})=\sum\limits_{\tau \in C_{k}} \sign(\tau)C(\tau)F_{\psi}(\overline{X}^\tau_{p,k}).
\end{equation*}
\end{prop}
\subsection{Some useful lemmas}
The following lemma exhibits the relationship between $F_{\psi}(\overline{X}^\tau_{p,k})$ and the cycle structure of $\tau$.  
\begin{lem}\label{Ftau}
Let $\tau \in C_{k}$ be the representative whose cyclic structure is associated with the partition $(1^{c_{1}},2^{c_{2}},\ldots k^{c_{k}})$ of $k$. Then we have $F_{\psi}(X^{\tau}_{p,k})=\prod_{i=1}^{k}(\sum\limits_{a \in D_p}\psi^{i}(a))^{c_{i}}$.
\end{lem} 
\begin{proof}
Recall that
\begin{align*}
F_{\psi}(X^{\tau}_{p,k})&=\sum\limits_{\bar{x} \in X_{p,k}^{\tau}} \prod_{i=1}^{k}\psi(x_{i})\\
&=\sum\limits_{\bar{x}\in X_{p,k}^{\tau}} \prod_{i=1}^{c_{1}}\psi(x_{i})\prod_{i=1}^{c_{2}}\psi^{2}(x_{c_{1}+2i})\ldots \prod_{i=1}^{c_{k}}\psi^{k}(x_{c_{1}+c_{2}\ldots+ki})\\
&= \prod_{i=1}^{k}(\sum\limits_{a \in D_p}\psi^{i}(a))^{c_{i}}.
\end{align*}
\end{proof}
Given $\chi\in \hat{G}$ define
\begin{eqnarray}\label{schi}
s_{D_p}(\chi):=\sum\limits_{a\in D_p} \chi(a).
\end{eqnarray}
Let $N(c_{1},c_{2},\ldots c_{k})$ denote the number of elements of $S_{k}$ of cycle type $(c_{1},c_{2},\ldots c_{k})$. It is well-known (see, for example, \cite{Stan}) that
\begin{eqnarray}\label{Ncom}
N(c_{1},c_{2},\ldots c_{k})=\dfrac{k!}{1^{c_1}c_1!2^{c_2}c_2!\cdots k^{c_k}c_k!}.
\end{eqnarray}
Then
\begin{lem}\label{Hpsi}
We have
\begin{eqnarray*}
F_{\psi}(\overline{X}_{p,k})=(-1)^k\sum_{\sum_{i}ic_i=k}N(c_1,c_2,\cdots, c_k)\prod_{i=1}^{k}(-s_{D_p}(\psi^i))^{c_{i}}.
\end{eqnarray*}
\end{lem}
\begin{proof}
To prove this lemma, we first note that $\sign(\tau)=(-1)^{k-\sum_{i}c_i}$. Also the cyclic structure for every $\tau\in C_k$ can be associated to a partition of $k$ of the form $(1^{c_1}, 2^{c_2},\cdots,k^{c_k})$. Hence the right-hand sum in Proposition \ref{LWS} runs over all such partitions of $k$. Noting that the conjugate permutations have same cycle type, and there are exactly $N(c_1,c_2,\cdots,c_k)$ permutations with cycle type $(c_{1},c_{2},\cdots c_{k})$ we conclude, in view of Lemma \ref{Ftau} that
\begin{eqnarray*}
F_{\psi}(\overline{X}_{p,k})=(-1)^k\sum_{\sum_{i}ic_i=k}N(c_1,c_2,\cdots, c_k)\prod_{i=1}^{k}(-\sum\limits_{a \in D_p}\psi^{i}(a))^{c_{i}}.
\end{eqnarray*}
\end{proof}
define the following polynomial in $k$ variables:
\begin{eqnarray}\label{Zgen}
Z_{k}(t_{1},\ldots, t_{k}):= \sum\limits_{\sum ic_{i}=k}N(c_{1},\ldots,c_{k})t_{1}^{c_{1}}\ldots t_{k}^{c_{k}}.
\end{eqnarray}
From Lemma \ref{Hpsi} and \eqref{Zgen} we immediately see that
\begin{cor}\label{HZ}
We have
\begin{eqnarray*}
F_{\psi}(\overline{X}_{p,k})=(-1)^kZ_k(-s_{D_p}(\psi),-s_{D_p}(\psi^2),\cdots,-s_{D_p}(\psi^k))
\end{eqnarray*}
where for $\chi\in\hat{G}$, $s_{D_p}(\chi)$ is as in \emph{(\ref{schi})}.
\end{cor}
Thus, it only remains to evaluate the sums $s_{D_p}(\chi)$ for $\chi=\psi^i, i=1, 2,\cdots, k$, and we do this next. Let $o(\chi)$ denotes the order of the character $\chi$. Then 
\begin{lem}\label{HZdel}
Let 
$$\delta_{1}^{\psi}(i):=\left\{\begin{array}{cc}
0,& \emph{if $o(\psi)\nmid i$},\\
-(p-1)N_p/p,& \emph{otherwise}
 \end{array}\right.$$
 and
 $$\delta_{2}^{\psi}(i):=\left\{\begin{array}{ccc}
0,&\emph{if $o(\psi)\neq 1,p$},\\
N_p/p,&\emph{if $o(\psi)/p \mid i$ and $o(\psi) \nmid i$},\\
-(p-1)N_p/p,& \emph{if $o(\psi) \mid i.$}
 \end{array}\right.$$ Then 
 \begin{enumerate}
 \item if $p \nmid o(\psi)$, $F_{\psi}(\overline{X}_{p,k})=(-1)^{k}Z_{k}(\delta_{1}^{\psi}(1),\ldots,\ \delta_{1}^{\psi}(k))$, and
 \item if $p \mid o(\psi)$, $F_{\psi}(\overline{X}_{p,k})=(-1)^{k}Z_{k}(\delta_{2}^{\psi}(1),\ldots,\ \delta_{2}^{\psi}(k))$.
 \end{enumerate}
\end{lem}  
\begin{proof}
First, observe that $G\setminus D_p$ is a subgroup of index $p$. 
Hence from elementary character theory, we can deduce that 
\begin{itemize}
    \item[A.] if $o(\psi)\neq 1,p$, we have $s_{D_p}(\psi)=s_{G}(\psi)-s_{G\setminus D_p}(\psi)=0$,
     \item[B.] if $o(\psi)=1$, we have $s_{D_p}(\psi)=|D_p|=(p-1)N_p/p$, and
    \item[C.] if $o(\psi)=p$, we have $s_{D_p}(\psi)=-s_{G\setminus D_p}(\psi)=-|G\setminus D_p|=-N_p/p$.
\end{itemize}
In order to estimate $F_{\psi}(\overline{X}_{p,k})$, we need to consider the following two cases:\\ \\
\textbf{Case I}: $p \nmid o(\psi)$. In this case, for all $i$ we have $p \nmid o(\psi^{i})$ since $o(\psi^i)=o(\psi)/(o(\psi),i)$. Thus from (A) and (B) we see that
\begin{eqnarray}
s_{D_p}(\psi^i)=\left\{\begin{array}{cc}
(p-1)N_p/p, & \text{if}\;o(\psi^i)=1,\\
0,& \text{otherwise}
\end{array}\right.
\end{eqnarray}
which implies (1) in view of Corollary \ref{HZ} and the definition of $\delta^\psi_1(i)$.\\ \\
\textbf{Case II}: $p\mid o(\psi)$. Here we have the following from (A), (B) and (C):
\begin{eqnarray}
s_{D_p}(\psi^i)=\left\{\begin{array}{ccc}
(p-1)N_p/p, & \text{if}\;o(\psi^i)=1,\\
-N_p/p,& \text{if}\;o(\psi^i)=p,\\
0,&\text{otherwise}
\end{array}\right.
\end{eqnarray}
which implies (2) in view of Corollary \ref{HZ} and the definition of $\delta^\psi_2(i)$.
\end{proof}
\subsection{Some combinatorial functions and estimates}
We now evaluate $Z_{k}\left(\delta_{1}^{\psi}(1),\cdots,\ \delta_{1}^{\psi}(k)\right)$ and\\ $Z_{k}(\delta_{1}^{\psi}(1),\ldots,\ \delta_{1}^{\psi}(k))$.
From (\ref{Ncom}) and (\ref{Zgen}) we immediately deduce the following:
\begin{lem}[Exponential generating function]\label{Egf}
We have
\begin{eqnarray*}
\sum\limits_{k \geq 0}Z_{k}(t_{1},t_{2}, \ldots t_{k})\dfrac{u^{k}}{k!}=e^{ut_{1}+u^{2}\frac{t_{2}}{2}+\cdots }
\end{eqnarray*}
\end{lem}
The next result follows by substituting special values for the variables $t_1, t_2, \cdots$ in Lemma \ref{Egf}. 
\begin{cor}\label{speZ}
We have
\begin{enumerate}
    \item if $t_i=a$ iff $d\mid i$ and $t_i=0$ iff $d\nmid i$, then
\begin{eqnarray*}
Z_k(\underbrace{0,\cdots,0}_{d-1},a,\underbrace{0,\cdots,0}_{d-1},a,\cdots)=\left[\dfrac{u^k}{k!}\right]\dfrac{1}{(1-u^d)^{a/d}}.
\end{eqnarray*}
\item if $t_i=a$ iff $d\mid i$ and $p\cdot d\nmid i$; if $t_i=b$ iff $p\cdot d\mid i$; and if $t_i=0$ iff $d\nmid i$, then
\begin{eqnarray*}
Z_k(\overbrace{\underbrace{0,\cdots,0}_{d-1},a,\underbrace{0,\cdots,0}_{d-1},a,\underbrace{0,\cdots,0}_{d-1}}^{p\cdot d-1},b,\cdots)=\left[\dfrac{u^k}{k!}\right]\dfrac{1}{(1-u^d)^{a/d}(1-u^{pd})^{\frac{b-a}{pd}}}.
\end{eqnarray*}
\end{enumerate}
\end{cor}
\begin{proof}
The proof of this corollary is similar to the case for $p=3$ in \cite[pp. 7, Lemma 2.3]{Li}.
\end{proof}
From Lemma \ref{HZdel} and Corollary \ref{speZ} we obtain:
\begin{lem}\label{Hfines}
We have
\begin{enumerate}
\item if $p\nmid o(\psi)$,
\begin{eqnarray*}
F_{\psi}(\overline{X}_{p,k})=(-1)^k\left[\dfrac{u^k}{k!}\right](1-u^{o(\psi)})^{\frac{(p-1)N_p}{po(\psi)}}\end{eqnarray*}
\item if $p|o(\psi)$,
\begin{eqnarray*}
F_{\psi}(\overline{X}_{p,k})=(-1)^k\left[\dfrac{u^k}{k!}\right]\dfrac{(1-u^{o(\psi)})^{\frac{N_p}{o(\psi)}}}{(1-u^{o(\psi)/p})^{\frac{N_p}{o(\psi)}}}.
\end{eqnarray*}
\end{enumerate}
\end{lem}
\section{Proofs of the main results}\label{Ps}
\begin{proof}[Proof of Theorem \emph{\ref{main1}}]
From (\ref{rc}) we have 
\begin{eqnarray}\label{Mkb}
M_{p,s,n}(k_1,k_2,\cdots,k_s;b)&=&N_p^{-1}\left\{\prod_{i=1}^s\binom{(p-1)N_p/p}{k_i}+P_{k_1,k_2,\cdots,k_s}+Q_{k_1,k_2,\cdots,k_s}+R_{k_1,k_2,\cdots,k_s}\right\},
\end{eqnarray}
where
\begin{align}\label{I-II-III}
P_{k_1,k_2,\cdots,k_s}&= \dfrac{1}{k_1!k_2!\cdots k_s!}\sum\limits_{\psi,p\nmid o(\psi)}\psi^{-1}(b)\prod_{i=1}^sF_{\psi}(\overline{X}_{p,k_i})\notag\\
Q_{k_1,k_2,\cdots,k_s}&=\dfrac{1}{k_1!k_2!\cdots k_s!}\sum\limits_{\psi, o(\psi)=p}\psi^{-1}(b)\prod_{i=1}^sF_{\psi}(\overline{X}_{p,k_i})\notag\\
R_{k_1,k_2,\cdots,k_s}&=\dfrac{1}{k_1!k_2!\cdots k_s!}\sum\limits_{ \substack{\psi,p\mid o(\psi)\\o(\psi)\neq p}}\psi^{-1}(b)\prod_{i=1}^sF_{\psi}(\overline{X}_{p,k_i}).
\end{align}
Using Lemma \ref{Hfines}, we see that
\begin{eqnarray}\label{I-II-III-r}
P_{k_1,k_2,\cdots,k_s}&=& \dfrac{(-1)^{k_1+k_2+\cdots+k_s}}{k_1!k_2!\cdots k_s!}\sum\limits_{\psi, p\nmid o(\psi)}\psi^{-1}(b) \prod_{i=1}^s\left[\dfrac{u^{k_i}}{k_i!}\right](1-u^{o(\psi)})^{\frac{(p-1)N_p}{po(\psi)}},\notag\\
Q_{k_1,k_2,\cdots,k_s}&=&\dfrac{(-1)^{k_1+k_2+\cdots+k_s}}{k_1!k_2!\cdots k_s!}\sum\limits_{\psi, o(\psi)=p}\psi^{-1}(b) \prod_{i=1}^s\left[\dfrac{u^{k_i}}{k_i!}\right]\left(\sum\limits_{i=0}^{p-1}u^{i}\right)^{\frac{N_p}{p}},\notag\\
R_{k_1,k_2,\cdots,k_s}&=&\dfrac{(-1)^{k_1,k_2,\cdots,k_s}}{k_1!k_2!\cdots k_s!}\sum\limits_{\substack{\psi, p\mid o(\psi)\\ o(\psi)\neq p}}\psi^{-1}(b) \prod_{i=1}^s\left[\dfrac{u^{k_i}}{k_i!}\right]\left(\sum\limits_{i=0}^{p-1}u^{o(\psi)i/p}\right)^{\frac{N_p}{o(\psi)}}.\notag\\
\end{eqnarray}
Recall that 
\begin{eqnarray}\label{Mb}
\ \ \ \ \ \ \ \ \ \ M_{p,s,n}(b)=\sum_{0\leq k_1,k_2,\cdots,k_s\leq |D_{p}|}(-1)^{k_1+k_2+\cdots+k_s}M_{p,s,n}(k_1,k_2,\cdots,k_s;b).
\end{eqnarray}
Using the well-known fact
\begin{eqnarray*}
\sum\limits_{k=0}^{|D_p|} (-1)^{k}\binom{|D_p|}{k}=0,
\end{eqnarray*}
we see that 
\begin{eqnarray}\label{pbino}
\mbox{}\\
\sum_{0\leq k_1,\cdots,k_s\leq |D_p|}(-1)^{k_1+k_2+\cdots+k_s}\prod_{i=1}^s\binom{(p-1)N_p/p}{k_i}=\left(\sum\limits_{k=0}^{|D_p|} (-1)^{k}\binom{|D_p|}{k}\right)^s=0.\notag
\end{eqnarray}
Thus (\ref{Mkb}), (\ref{Mb}) and (\ref{pbino}) yield,
\begin{eqnarray}\label{Mbr}
M_{p,s,n}(b)&=&\dfrac{1}{N_p}\left\{\sum_{0\leq k_1,k_2,\cdots,k_s\leq |D_{p}|}(-1)^{k_1+k_2+\cdots+k_s}P_{k_1,k_2,\cdots,k_s}\right.\notag\\&&\hspace{1.5 cm}\left.+\sum_{0\leq k_1,k_2,\cdots,k_s\leq |D_{p}|}(-1)^{k_1+k_2+\cdots+k_s}Q_{k_1,k_2,\cdots,k_s}\right.\notag\\&&\hspace{2 cm}\left.+\sum_{0\leq k_1,k_2,\cdots,k_s\leq |D_{p}|}(-1)^{k_1+k_2+\cdots+k_s}R_{k_1,k_2,\cdots,k_s}\right\}.
\end{eqnarray}
Given a character $\psi$ of order $p$, there is a unique $x \in \{1,\cdots,p-1\}$ such that for all $y \in \mathbb{Z}_{N_p}$, we have $\psi(y)=e^{2\pi i xy/p}$. Now\\ \\ 
\textbf{Case I}: If $p \mid b$. Then 
\begin{eqnarray}\label{exps1}
\sum\limits_{\psi,o(\psi)=p}\psi^{-1}(b)=\sum\limits_{x=1}^{p-1} e^{2\pi i bx/p}=p-1
\end{eqnarray}
\mbox{}\\ \\
\textbf{Case II}: If $p \nmid b$. Then, as $x$ runs over elements in $Z_p^{\times}$, so does $bx$ and we get 
\begin{eqnarray}\label{exps2}
\sum\limits_{\psi,o(\psi)=p}\psi^{-1}(b)=\sum\limits_{x=1}^{p-1} e^{2\pi i bx/p}=\sum\limits_{x=1}^{p-1} e^{2\pi i x/p}=-1.
\end{eqnarray}
So from (\ref{I-II-III-r}) we have
\begin{align}\label{summulti}
&\sum_{0\leq k_1,k_2,\cdots,k_s\leq |D_{p}|}(-1)^{k_1+k_2+\cdots+k_s}Q_{k_1,k_2,\cdots,k_s}\notag\\&= \left(\sum\limits_{\psi,o(\psi)=p}\psi^{-1}(b)\right) \left(\sum\limits_{0\leq k_1,\cdots,k_s\leq |D_p|}\dfrac{1}{k_1!\cdots k_s!} \prod_{i=1}^s\left[\dfrac{u^{k_i}}{k_i!}\right]\left(\sum\limits_{j=0}^{p-1} u^{j}\right)^{\frac{N_p}{p}}\right)\notag\\
&= \left(\sum\limits_{\psi,o(\psi)=p}\psi^{-1}(b)\right) \left(\sum\limits_{0\leq k_1,k_2,\cdots,k_s\leq |D_p|}\prod_{i=1}^s\left[u^{k_i}\right]\left(\sum\limits_{j=0}^{p-1} u^{j}\right)^{\frac{N_p}{p}}\right).\notag\\
&=\left(\sum\limits_{\psi,o(\psi)=p}\psi^{-1}(b)\right)\left(\sum\limits_{k=0}^{|D_p|}\left[u^{k}\right]\left(\sum\limits_{j=0}^{p-1} u^{j}\right)^{\frac{N_p}{p}}\right)^s.
\end{align}
Noting that the sum
\begin{eqnarray}
\sum\limits_{k=0}^{|D_p|}\left[u^{k}\right]\left(\sum\limits_{j=0}^{p-1} u^{j}\right)^{\frac{N_p}{p}}=p^{N_p/p}
\end{eqnarray}
since it is the sum of all coefficients of the multinomial expansion of $(1+u+u^2+\cdots+u^{p-1})^{N_p/p}$, we obtain the following from (\ref{exps1}), (\ref{exps2}) and (\ref{summulti}):
\begin{eqnarray}\label{Rr}
\varSigma_{p,s,n,b}&:=&\sum_{0\leq k_1,k_2,\cdots,k_s\leq |D_{p}|}(-1)^{k_1+k_2+\cdots+k_s}Q_{k_1,k_2,\cdots,k_s}\notag\\&=&\left\{\begin{array}{cc}
(p-1)\cdot p^{sN_p/p},&\text{if $p \mid b$,}\\
-p^{sN_p/p},&\text{otherwise.}\end{array}\right.
\end{eqnarray}
Next, we estimate $P_{k_1,k_2,\cdots,k_s}$ and $R_{k_1,k_2,\cdots,k_s}$. Consider 
\begin{align}\label{Sr}
&\left|\sum_{0\leq k_1,k_2,\cdots,k_s\leq |D_{p}|}(-1)^{k_1+k_2+\cdots+k_s}R_{k_1,k_2,\cdots,k_s} \right|\notag\\
&=
\left|\sum\limits_{\substack{\psi, p\mid o(\psi)\\ o(\psi)\neq p}}\psi^{-1}(b) \sum_{0\leq k_1,k_2,\cdots,k_s\leq |D_p|}\prod_{i=1}^s\dfrac{1}{k_i!}\left[\dfrac{u^{k_i}}{k_i!}\right]\left(\sum\limits_{j=0}^{p-1}u^{o(\psi)j/p}\right)^{\frac{N_p}{o(\psi)}}\right|\notag\\
&=
\left|\sum\limits_{\substack{\psi,p \mid o(\psi)\\ o(\psi)>p}}\psi^{-1}(b)\sum\limits_{0\leq k_1,k_2,\cdots,k_s\leq |D_p|}\prod_{i=1}^s\left[{u^{k_i}}\right]\left(\sum\limits_{j=0}^{p-1}u^{o(\psi)j/p}\right)^{\frac{N_p}{o(\psi)}}\right|\notag\\
&\leq \left|\sum\limits_{\substack{p \mid o(\psi)\\o(\psi)>p}}\left(\sum_{k=0}^{|D_p|}[u^k]\left(\sum_{j=0}^{p-1}u^{o(\psi)j/p}\right)^{N_p/o(\psi)}\right)^s\right|\notag\\
&\leq \sum\limits_{\substack{p \mid o(\psi)\\o(\psi)>p}} p^{sN_p/o(\psi)}=N_p\cdot p^{sN_p/2p}.
\end{align}
Finally, we consider
\begin{eqnarray*}
&&\left|\sum_{0\leq k_1,k_2,\cdots,k_s\leq |D_{p}|}(-1)^{k_1+k_2+\cdots+k_s}P_{k_1,k_2,\cdots,k_s} \right|\notag\\&&=\left|
\sum\limits_{\psi, p\nmid o(\psi)}\psi^{-1}(b)\sum_{0\leq k_1,k_2,\cdots,k_s\leq |D_{p}|} \prod_{i=1}^s\dfrac{1}{k_i!}\left[\dfrac{u^{k_i}}{k_i!}\right](1-u^{o(\psi)})^{\frac{(p-1)N_p}{po(\psi)}}\right|\notag
\end{eqnarray*}
\begin{eqnarray}\label{Qr}
&&= \left|\sum\limits_{\substack{\psi,p\nmid o(\psi)\\ o(\psi)>1}}\psi^{-1}(b)\sum_{0\leq k_1,k_2,\cdots,k_s\leq |D_{p}|} \prod_{i=1}^s\left[u^{k_i}\right](1-u^{o(\psi)})^{\frac{(p-1)N_p}{po(\psi)}}\right|\notag\\
&&=\left|\sum\limits_{\substack{\psi,p\nmid o(\psi)\\ o(\psi)>1}}\psi^{-1}(b)\left(\sum_{k=0}^{|D_p|} \left[u^{k}\right](1-u^{o(\psi)})^{\frac{(p-1)N_p}{po(\psi)}}\right)^s\right|\notag\\
&&=0,
\end{eqnarray}
where the last step is obtained by noting that $\sum\limits_{k=0}^{|D_{p}|}[u^{k}](1-u^{j})^{\frac{(p-1)N_p}{pj}}$ is the sum of all coefficients of $(1-u^{j})^{\frac{(p-1)N_p}{pj}}$ which is zero. 

Hence (\ref{Mbr}), (\ref{Rr}), (\ref{Sr}) and (\ref{Qr}) yield
\begin{align*}
\left|M_{p,s,n}(b)-\dfrac{\varSigma_{p,s,n,b}}{N_p}\right| &\leq N_p^{-1} \left|\sum_{0\leq k_1,k_2,\cdots,k_s\leq |D_{p}|}(-1)^{k_1+k_2+\cdots+k_s}P_{k_1,k_2,\cdots,k_s} \right|\notag\\&+ N_p^{-1}\left|\sum_{0\leq k_1,k_2,\cdots,k_s\leq |D_{p}|}(-1)^{k_1+k_2+\cdots+k_s}R_{k_1,k_2,\cdots,k_s} \right|\notag\\
& \leq p^{sN_p/2p},
\end{align*}
which yields the theorem.
\end{proof}
\begin{proof}[Proof of Theorem \emph{\ref{main2}}]
To prove this theorem, we need to consider two cases.\\ \\
\textbf{Case I}: $p\mid b$. In this case, from Theorem \ref{main1} we have 
\begin{eqnarray*}
\varSigma_{p,s,n,b}=(p-1)\cdot p^{s(n+1)}.
\end{eqnarray*}
Thus the main term in Theorem \ref{main1} dominates the error provided
\begin{eqnarray*}
\dfrac{(p-1)\cdot p^{s(n+1)-1}}{n+1}>p^{s(n+1)/2}
\end{eqnarray*}
It is now clear that for all $n\geq \inf\{n\in\mathbb{N}: (p-1)p^{s(n+1)/2-1}>n+1\}$, $M_{p,s,b}(b)>0$ when $p\mid b$. \\ \\
\textbf{Case II}: $p\nmid b$. In this case, from Theorem \ref{main1} we have 
\begin{eqnarray*}
\varSigma_{p,s,n,b}=-p^{s(n+1)}.
\end{eqnarray*}
Thus the absolute value of the main term in Theorem \ref{main1} dominates the error provided
\begin{eqnarray*}
\dfrac{p^{s(n+1)-1}}{n+1}>p^{s(n+1)/2}.
\end{eqnarray*}
It is now clear that for all $n\geq \inf\{n\in\mathbb{N}: p^{s(n+1)/2-1}>n+1\}$, $M_{p,s,n}(b)<0$ when $p\nmid b$. This proves the theorem. 
\end{proof}
\begin{proof}[Proof of Theorem \emph{\ref{main3}}]
The first part of the theorem follows from Theorem \ref{main1} by choosing $p=3$ and $s=1$. For the other part, we use Theorem \ref{main2}. Thus the smallest $n_{3,1,b}\in\mathbb{N}$ for which
\begin{eqnarray*}
3^{(n+1)/2-1}>n+1
\end{eqnarray*}
holds true is $n_{3,1,b}=4$. Thus for all $n\geq 4$ we have $M_{3,1,n}(b)<0$. Also by direct computation, one shows that $M_{3,1,n}(b)<0$ for all $n<4$. Indeed, using Wang's result \cite{Wang}, one immediately concludes that $M_{3,1,n}(b)<0$ without any of the above analysis.
\end{proof}
\begin{proof}[Proof of Theorem \emph{\ref{main4}}]
The first part of this theorem follows directly from Theorem \ref{main1} by choosing $p=3$ and $s=2$. 

For the other part, we use Theorem \ref{main2}. Thus in the case $b\equiv 0 \pmod 3$ we have
\begin{eqnarray*}
2\cdot 3^{n}>n+1
\end{eqnarray*}
for all $n\in\mathbb{N}$. Hence $M_{3,2,n}>0$ for all $n\in\mathbb{N}$. In the case $b\not\equiv 0 \pmod 3$ we have
\begin{eqnarray}
3^{n}>n+1
\end{eqnarray}
holds true for all $n\in\mathbb{N}$. Hence $M_{3,2,n}<0$ for all $n\in\mathbb{N}$.
\end{proof}
\begin{proof}[Proof of Theorem \emph{\ref{main5}}]
The first part of this theorem follows directly from Theorem \ref{main1} by choosing $p=5$ and $s=1$. 

For the other part, we use Theorem \ref{main2}. Thus in the case $b\equiv 0 \pmod 5$ we have
\begin{eqnarray*}
4\cdot 5^{(n+1)/2-1}>n+1
\end{eqnarray*}
for all $n\in\mathbb{N}$. Hence $M_{5,1,n}>0$ for all $n\in\mathbb{N}$. In the case $b\not\equiv 0 \pmod 5$ we have
\begin{eqnarray}
5^{(n+1)/2-1}>n+1
\end{eqnarray}
holds true for all $n\geq 3$. By direct computation one checks that $M_{5,1,n}<0$ for all $n<3$. Hence $M_{5,1,n}<0$ for all $n\in\mathbb{N}$.
\end{proof}
\begin{proof}[Proof of Theorem \emph{\ref{Borw1}}]
Using Cauchy's formula we see that 
 \begin{eqnarray}\label{maxpeak11}
 |t_{j,p,s}|=\left|\dfrac{1}{2i\pi}\int_{|q|=1}\dfrac{T_{p,s,n}(q)}{q^{j+1}}dq\right|\leq \dfrac{1}{2\pi}\max_{|q|=1}|T_{p,s,n}(q)|\int_{|q|=1}\dfrac{|dq|}{|q|^{j+1}}=\max_{|q|=1}|T_{p,s,n}(q)|.
 \end{eqnarray}
 On the other hand we have
 \begin{eqnarray}\label{maxpeak22}
 \max_{|q|=1}|T_{p,s,n}(q)|\leq \sum_{0\leq j\leq N_{s,n}}|t_{j,p,s}|\leq (d_{p,s,n}+1)\max_{j}|t_{j,p,s}|. 
 \end{eqnarray}
 Since $d_{p,s,n}\leq sp^2n^2$, (\ref{maxpeak11}) and (\ref{maxpeak22}) imply
 \begin{eqnarray}\label{Borwe1}
 \log_p \max_{j}|t_{j,s,n}|=\log_p \max_{|q|=1}|T_{p,s,n}(q)|+O(\log n).
 \end{eqnarray}
 Thus the theorem follows if we show that 
 \begin{eqnarray}\label{Borwe2}
 \log_p \max_{|q|=1}|T_{p,s,n}(q)|=s(n+1)+O(\log n).
 \end{eqnarray}
 We note that
 \begin{eqnarray}\label{Borwe3}
 \max_{|q|=1}|T_{p,s,n}(q)|=\max_{|q|=1}\left|\prod_{j=1}^n(1-q^j)\right|^s=\left(\max_{|q|=1}\left|\prod_{j=0}^n\prod_{k=1}^{p-1}(1-q^{pj+k})\right|\right)^s.
 \end{eqnarray}
 From \cite[Theorem 1, pp. 229]{Bor}, we have 
 \begin{eqnarray}\label{Borwe4}
 \log_p\max_{|q|=1}\left|\prod_{j=0}^n\prod_{k=1}^{p-1}(1-q^{pj+k})\right|=n+1+O\left(\dfrac{1}{n}\right).
 \end{eqnarray}
Now the estimate (\ref{Borwe2}) and thus the theorem follow from (\ref{Borwe1}), (\ref{Borwe3}) and (\ref{Borwe4}).
\end{proof} 
\begin{proof}[Proof of Theorem \emph{\ref{Borw2}}]
We have
\begin{eqnarray}
 \max_{|q|=1}|T_{p,s,n}(q)|\leq \sum_{0\leq j\leq N_{s,n}}|t_{j,p,s}|\leq (d_{p,s,n}+1)\max_{j}|t_{j,p,s}|,
\end{eqnarray}
which implies since $d_{p,s,n}\leq sp^2n^2$ that
\begin{eqnarray}\label{maxpeak222}
 \max_{j}|t_{j,p,s}|\geq \dfrac{1}{d_{p,s,n}+1}\max_{|q|=1}|T_{p,s,n}(q)|>\dfrac{1}{sp^2n^2}\max_{|q|=1}|T_{p,s,n}(q)|.
\end{eqnarray}
From \cite[Theorem 2, pp. 229]{Bor} we have
\begin{eqnarray}\label{Bo23}
\max_{|q|=1}\left|\prod_{j=0}^n\prod_{k=1}^{p-1}(1-q^{pj+k})\right|\gtrsim (1.219\cdots)^{(p-1)(n+1)}>p^{n+1}. 
\end{eqnarray}
Since 
\begin{eqnarray}
\max_{|q|=1}|T_{p,s,n}(q)|=\left(\max_{|q|=1}\left|\prod_{j=0}^n\prod_{k=1}^{p-1}(1-q^{pj+k})\right|\right)^s,
\end{eqnarray}
using (\ref{Bo23}) in (\ref{maxpeak222}) yields
\begin{eqnarray}
\max_{j}|t_{j,p,s}|\gtrsim \dfrac{(1.219\cdots)^{s(p-1)(n+1)}}{sp^2n^2}>\dfrac{p^{s(n+1)-2}}{sn^2}.
\end{eqnarray}
\end{proof}

\end{document}